\documentclass[11pt,regno]{amsart}
    \usepackage[margin=3.8cm]{geometry}
\numberwithin{equation}{section}
\usepackage{amssymb}
\usepackage{amsmath, amsfonts,amsthm,amssymb,amscd, verbatim,graphicx,color,multirow,booktabs, caption,tikz,tikz-cd, mathdots,bm}
\usepackage{tikz-cd}
 \usepackage{hyperref} 
\usetikzlibrary{positioning} 

\newtheorem*{theorem*}{Theorem}
\newtheorem*{lemma*}{Lemma}
\newtheorem*{proposition*}{Proposition}
\newtheorem*{problem*}{Problem}

     \newtheorem*{question*}{Question}

      \theoremstyle{definition}

     \theoremstyle{remark}
     \newtheorem*{remark*}{Remark}

\newcommand{\Sym}{\mathop{\mathrm{Sym}}}

\newcommand{\PSL}{\mathop{\mathrm{PSL}}}

\newcommand{\lcm}{\mathop{\mathrm{lcm}}}

\newcommand{\expr}{\mathcal{E}}

 \definecolor{mycolor}{rgb}{0.55,0.0,0.16}
  \definecolor{myred}{rgb}{0.75,0.0,0.16} 
  \definecolor{mygreen}{rgb}{0.0,0.4,0.16} 
  \definecolor{myviolet}{rgb}{1,0,1} 
   \definecolor{mypink}{rgb}{0.67,0,0.47}

\hypersetup{
colorlinks=true,
linkcolor=true,
linktocpage=true,
pageanchor=true,
hyperindex=true
}

 \AtBeginDocument{
     \hypersetup{
  linkcolor=mycolor,
  urlcolor=mypink,
citecolor=mygreen
}
     }

\makeatletter
\@namedef{subjclassname@2020}{%
  \textup{2020} Mathematics Subject Classification}
\makeatother


\author[Luca Sabatini]{Luca Sabatini}
\address{\parbox{\linewidth}{Luca Sabatini, Mathematics Institute, Zeeman Building, University of Warwick\\
Coventry CV4\,7AL, United Kingdom \vspace{0.1cm}}}
\email{luca.sabatini@warwick.ac.uk, sabatini.math@gmail.com} 

\begin{document} 
 \title[Exponent and number of generators]{Exponent and number of generators\\in a finite group} 

\maketitle 

\begin{abstract} 
We present
a sharp upper bound for the number of generators of a finite group
in terms of the ratio between the order and the exponent.
\end{abstract} 

\vspace{1cm}

The exponent of a finite group is the smallest natural number $n$ such that $x^n =1$ for all elements $x$.
It is always a divisor of the order of the group, with equality if and only if all the Sylow subgroups are cyclic.
A classical theorem in finite group theory \cite[Th. 5.16]{Isa08} says that
a group with this property is metacyclic, so in particular it can be generated by $2$ elements.

In 1989, Guralnick and Lucchini \cite{Gur89,Luc89} have greatly generalized this fact by showing that,
if all the Sylow subgroups can be generated by $d$ elements,
then the whole group can be generated by $d+1$ elements.
In this note we prove the following inequality, which improves the classical theorem in a different direction
related to the exponent.
We write $d(G)$ for the minimal number of generators and $\exp(G)$ for the exponent.

\begin{theorem*} \label{thMain}
For every finite group $G$ we have
$$ p^{d(G) -2} \> \leq \> \frac{|G|}{\exp(G)} , $$
where $p$ is the smallest prime dividing $|G|$.

Equality holds if and only if $G$ is a noncyclic group with $\exp(G)=|G|$.
\end{theorem*}

This is the first upper bound for the number of generators in terms of the ratio between the order and the exponent.
The proof we present here relies on \cite{Gur89,Luc89}, 
which unfortunately depends on the Classification of the Finite Simple Groups.
 It would be very interesting to find a direct proof.\\

{\bfseries 1. The nilpotent case.}
Let $p$ be the smallest prime dividing $|G|$.
Moreover, let us write $\expr(G) = \frac{|G|}{\exp(G)}$.
The connection between $d(G)$ and $\expr(G)$ is transparent when $G$ is nilpotent.

\begin{lemma*} \label{lem1}
Let $G$ be a finite nilpotent group. Then $p^{d(G)-1} \leq \expr(G)$.
\end{lemma*}
\begin{proof}
The key point is that nilpotency guarantees the existence of an element $x$ of order $\exp(G)$.
This is clear for a $p$-group,
and the general case follows because elements of coprime order commute.
Now the subgroup generated by $x$ has index $\expr(G)$ in $G$.
If $x_1=x,\ldots,x_n$ are elements of $G$ with the property that $x_{i+1} \notin \langle x_1,\ldots,x_i \rangle$ for each $i$,
then $|G: \langle x_1,\ldots,x_n \rangle| \leq \frac{\expr(G)}{p^{n-1}}$.
The proof follows.
\end{proof}

In general, the exponent can be dramatically larger than the maximum order of an element.
In the symmetric group on $n$ points, the first is $\lcm(1,\ldots,n)$, which is asymptotically $e^{(1+o(1)) n}$,
while the second is just $e^{(1+o(1)) \sqrt{n \ln n}}$ \cite{Lan03}.\\

{\bfseries 2. The general case.}
Let $(P_i)_{i \geq 1}$ be a full set of Sylow subgroups of $G$ for distinct primes.
Then $|G| = \prod_{i \geq 1} |P_i|$.
Moreover, it is easy to see that $\exp(G) = \prod_{i \geq 1} \exp(P_i)$.
It follows that
\begin{equation} \label{eq} 
\expr(G) \> = \> \prod_{i \geq 1} \expr(P_i) . \tag{*}
\end{equation}
If $N \unlhd H \leqslant G$, then we say that $H/N$ is a {\itshape section} of $G$.
For any prime $p$, the Sylow $p$-subgroup of a section of $G$ is a section of the Sylow $p$-subgroup of $G$.
Therefore, it follows from (\ref{eq}) that
\begin{equation} \label{eqD} 
\expr(H/N) \mbox{ is always a divisor of } \expr(G) . \tag{**} 
\end{equation} 

\begin{proof}[Proof of the Theorem]
Let $P_i$ be a Sylow $p_i$-subgroup for each $i$.
Combining (\ref{eq}) with the Lemma, we obtain
$$ \expr(G) = \prod_{i \geq 1} \expr(P_i) \geq \prod_{i \geq 1} p_i^{d(P_i)-1} . $$
Passing to logarithms, it remains to show that 
\begin{equation} \label{eqq} 
\sum_{i \geq 1} \frac{(d(P_i)-1) \log p_i}{\log p} \> \geq \> d(G) -2 . \tag{***}
\end{equation}
Suppose $d(P_1) \geq d(P_i)$ for all $i$.
By \cite{Gur89,Luc89} we have $d(P_1) \geq d(G)-1$, and this proves the inequality.
If equality holds in (\ref{eqq}),
then $d(P_1) = d(G)-1$, $p=p_1$, and $P_i$ is cyclic for all $i \geq 2$.
It follows from \cite[Cor. 4]{Luc97} that also $P_1$ has to be cyclic and we are done.
\end{proof}

\begin{remark*}
The Theorem implies that if $\expr(G)=2$, then $d(G)=2$.
This is attained by $\PSL_2(p)$ for $p$ odd.
Moreover, if $G= (C_3)^n \rtimes C_2$ with $C_2$ acting as a power automorphism,
then $d(G)=n+1$ and $\expr(G)=3^{n-1}$.
\end{remark*}


In a private communication,
Lucchini observed that $d(G) = O(\log \log \expr(G))$ holds in several cases.
In fact, $d(G)$ equals $d(L_k)$ where $L_k$ is a crown-based power
of the monolithic primitive group $L$ with socle $N \cong T^t$ \cite[Th. 3]{Luc23}.
Moreover, $L_k$ is isomorphic to a quotient of $G$ and contains a normal subgroup isomorphic to $N^k$.
If $N$ is nonabelian, then
$$ d(G) = d(L_k) = O( \log (tk)) . $$
On the other hand $\expr(G) \geq \expr(T^{tk}) = |T|^{tk-1} \expr(T)$ (because $\exp(T^{tk})=\exp(T)$),
and so
$$ tk = O \left( \log \left( \frac{\expr(G)}{\expr(T)} \right) \right) . $$ \\

{\bfseries 3. A direct proof\,?}
While an elementary proof of \cite{Gur89,Luc89} seems to be out of reach,
proving the Theorem (or some weak version of it) might be more approachable.
The central problem is the following:

\begin{problem*}
Prove (avoiding CFSG) that $d(G)$ is bounded by a function on $\expr(G)$.
\end{problem*}

A finite group is {\itshape $d$-maximal} if it requires more generators than every proper subgroup.
For a comprehensive work on these groups, we refer the reader to \cite{LSS24}.
From (\ref{eqD}), we can assume that $G$ is $d$-maximal
and $\expr(M)=\expr(G)$ for all maximal subgroups $M$ of $G$.
The case $\expr(G)=2$ follows from the structural result of Wong \cite{Won66}.
Moreover, the next fact is essentially well known.

\begin{proposition*}
If $\expr(G)$ is odd, then $G$ is solvable.
\end{proposition*}
\begin{proof}
The hypothesis means precisely that, if nontrivial, a Sylow $2$-subgroup of $G$ is cyclic, say generated by $x$.
This implies that the Cayley embedding of $G$ into $\Sym(|G|)$ has to map $x$ to an odd permutation.
It follows that $G$ has a characteristic subgroup of index $2$,
and by iteration that $G$ has a characteristic subgroup of odd order and index a power of $2$.
By the Feit-Thompson theorem such a subgroup is solvable and we are done.
\end{proof}

   \vspace{0.5cm}

\end{document}